\documentclass[11pt,english]{amsart}
\usepackage[T1]{fontenc}
\usepackage[latin9]{inputenc}
\usepackage{amsthm}
\usepackage{amstext}
\usepackage{stmaryrd}
\usepackage{esint}
\usepackage{mathrsfs}
\makeatletter
\numberwithin{equation}{section}
\numberwithin{figure}{section}
\theoremstyle{plain}
\newtheorem{thm}{\protect\theoremname}[section]
  \theoremstyle{plain}
  \newtheorem{lem}[thm]{\protect\lemmaname}

\usepackage{amsmath,amsfonts,amssymb,amsthm,epsfig}

\usepackage{color}

\usepackage[final,allcolors=blue,colorlinks=true]{hyperref}
\usepackage[leqno]{amsmath}

\newcommand{\ds}{\displaystyle}

\def\R{\mathbb R}

\def\al{\alpha}

\def\ga{\gamma}

\def\ep{\epsilon}
\def\la{\lambda}

\def\var{\varphi}

\def\Om{\Omega}

\def\pa{\partial}

\numberwithin{equation}{section}
\theoremstyle{definition}
\begingroup
\newtheorem{rem}[thm]{Remark}
\newtheorem{prop}[thm]{Proposition}

\endgroup

\makeatother

\usepackage{babel}
  \providecommand{\lemmaname}{Lemma}
\providecommand{\theoremname}{Theorem}

\begin{document}

\title{Infinitely many sign-changing solutions for an elliptic problem ]Infinitely many sign-changing solutions for an elliptic problem with double critical Hardy-Sobolev-Maz'ya
 terms}

\author{Chunhua Wang, Jing Yang}

\date{\today}

\address{[Chunhua Wang] School of Mathematics and Statistics, Central China
Normal University, Wuhan 430079, P. R. China. }

\email{[Chunhua Wang] chunhuawang@mail.ccnu.edu.cn}

\address{[Jing Yang] School of Mathematics and Statistics, Central China
Normal University, Wuhan 430079, P. R. China.}

\email{[Jing Yang] yyangecho@163.com}
\begin{abstract}
In this paper, we investigate the following elliptic problem involving double
critical Hardy-Sobolev-Maz'ya terms:
\begin{equation*}\label{0.1}
\left\{%
\begin{array}{ll}
    -\Delta u=\mu\frac{|u|^{2^{*}(t)-2}u}{|y|^{t}}+\frac{|u|^{2^{*}(s)-2}u}{|y|^{s}}+a(x) u, & \hbox{$\text{in}~ \Omega$},\vspace{0.1cm} \\
   u=0,\,\, &\hbox{$\text{on}~\partial \Omega$},
\end{array}%
\right.
\end{equation*}
where $\mu\geq0, a(x)>0, 2^{*}(t)=\frac{2(N-t)}{N-2},2^{*}(s)=\frac{2(N-s)}{N-2},0\leq
t<s<2,x=(y,z)\in \R^{k}\times \R^{N-k},2\leq k<N,(0,z^*)\in \bar{\Omega}$ and $\Omega$
is an bounded domain in $\R^{N}.$ Applying an abstract theorem in \cite{sz},
we prove that
if $N>6+t$ when $\mu>0,$ and $N>6+s$ when $\mu=0,$ and
$\Omega$ satisfies some geometric conditions,
then the above problem has infinitely many sign-changing solutions.
The main tool is to estimate Morse indices of these nodal solution.
\end{abstract}

\maketitle
{\small

\smallskip

\tableofcontents{}

\section{ Introduction and the main results}
 Let $N\geq 3,\mu\geq0, 2^{*}(t)=\frac{2(N-t)}{N-2},2^{*}(s)=\frac{2(N-s)}{N-2},0\leq t<s<2$  and $\Omega$ is
an open bounded domain in $\R^{N}$ satisfying some geometric conditions
and $0<a(x)\in C^{1}(\bar{\Omega}).$
 We consider the following elliptic problem with double critical Hardy-Sobolev-Maz'ya terms:
\begin{equation}\label{1.1}
\left\{%
\begin{array}{ll}
    -\Delta u=\frac{\mu|u|^{2^{*}(t)-2}u}{|y|^{t}}+\frac{|u|^{2^{*}(s)-2}u}{|y|^{s}}+a(x)u, & \hbox{$\text{in}~ \Omega$},\vspace{0.1cm} \\
   u=0,\,\, &\hbox{$\text{on}~\partial \Omega$}, \\
\end{array}%
\right.
\end{equation}
where $x=(y,z)\in \R^{k}\times \R^{N-k},2\leq k<N.$

The main interest in problem \eqref{1.1} relates to the following Hardy-Sobolev-Maz'ya inequality \cite{bs}: there exists a positive constant
$S_s(\Omega)$ such that
\begin{equation}\label{1.1.1}
\Big(\ds\int_\Omega\frac{|u|^{2^*(s)}}{|y|^{s}}dx\Big)^{\frac{2}{2^*(s)}}
\leq (S_s(\Omega))^{-1}\ds\int_\Omega|\nabla u|^2dx,\,\,\,\,\,u\in
H^{1}_{0}(\Omega).
\end{equation}

For the results on whether the optimal constant $S_s(\Omega)$ can be attained, one can refer to \cite{bs,ms,m} if $k<N$ and
 \cite{c,cw,cl} if $k=N$ and the references therein.
Due to the double Hardy-Sobolev-Maz'ya terms of problem \eqref{1.1},
in order to get a positive solution for \eqref{1.1}, one can
apply the idea of  Brezis and Nirenberg \cite{bn}, or the concentration compactness principle of Lions \cite{pl}, or the global compactness of Struwe
\cite{s}, to show that the mountain pass value is a critical value of the functional
\begin{equation}\label{1.2}
I(u)=\frac{1}{2}\int_{\Omega}\big(|\nabla u|^{2}-a(x)
u^{2}\big)dx-\frac{\mu}{2^{*}(t)}\int_{\Omega}\frac{|u|^{2^{*}(t)}}{|y|^{t}}dx
-\frac{1}{2^{*}(s)}\int_{\Omega}\frac{|u|^{2^{*}(s)}}{|y|^{s}}dx,\,\,\,u\in
H^{1}_{0}(\Omega).
\end{equation}

It is worth noting that the crucial step in the proof is to show that the mountain pass value is strictly less than the first energy level  where the
Palais Smale condition fails. For the existence of the mountain pass solution for \eqref{1.1}, one can see (\cite{awz,bs,mfs,cfms2,ms,ms1,m}).
And in the case $k=N$, one can refer to \cite{gk,gr,hl,ll,cp2,cp,GY} and the references therein.
We want to point out that in the case $0\in \partial \Omega$,
the mean curvature of $\partial \Omega$ at $0$ plays an important role in the existence of the mountain pass solution.
See for example \cite{gk,gr,hl,ll,cl}.

 Since the pioneer work of Brezis and Nirenberg \cite{bn} appeared, there are
 a number of results for problem \eqref{1.1} with $\mu=0,s=0,k=N$ which is called Brezis-Nirenberg type problem. See for example \cite{bc,cfp,css,s} and particularly the survey paper \cite{bn} and
 the references therein. Moreover, the study of sign-changing solution to some elliptic equations has been an increasing interest
 in recent years (cf.\cite{bcw, blw,bw,cmt1,cmt,lw,rtz,szw})and the references therein.
 In \cite {cp1}, Cao and Peng studied problem
 \eqref{1.1} with $t=2,s=0$. They obtained the existence of
 a pair of sign-changing solutions if $a(x)=\la>0$ and $N\geq 7$.
  In \cite{sz}, Schechter and Zou obtained that there exist infinitely many sign-changing solutions
  to the  Brezis-Nirenberg type problem in higher dimension$(N\geq 7)$.
  After the excellent work by Schechter and Zou \cite{sz}, Chen and Zou in \cite{cz1} obtained infinitely many sign-changing solutions for
problem \eqref{1.1} with the Hardy term in the case $t=2, s=0$ when $k=N.$
And very recently in \cite{g}, it was proved that \eqref{1.1} has infinitely many sign-changing solutions
 in the case $t=2,a(x)=\la>0$ and $0\in\Omega.$

Based on these results mentioned above, a natural and interesting question is whether \eqref{1.1} has infinitely many sign-changing solutions.
 As far as we know, there is no any information on that. But, in this paper, we give a positive answer and precisely, we state our
 main result as follows.

\begin{thm}\label{thm1.1}
Assume that $a(x)>0$ and
$\Omega$ is a bounded domain such that
$(x-(0,z^{*}))\cdot\nu\leq 0$ in a neighborhood of $(0,z^{*})$ if
 $(0,z^{*})\in \partial\Omega,$ where $\nu$ is the outward normal of $\partial\Omega.$
If $N>6+t$ when $\mu>0,$ and $N>6+s$ when $\mu=0,$ then \eqref{1.1} has infinitely many sign-changing solutions.
\end{thm}

\begin{rem}\label{rem1.3}
Our result includes the special case that when $k=N$ and $t=0$ in \eqref{1.1}
which studied by
Bhakta in \cite{mb}. He assumed that $\pa\Omega\in C^{3}$ and
if $0\in \partial\Omega,$ all the principle curvature of $\partial\Omega$
at $0$ are negative. However, in this case, our condition is much weaker than their's.
For the details, one can refer to \cite{wy}.
\end{rem}

Since
$2^{*}(t)$ and $2^{*}(s)$ are the double critical exponents for the Sobolev embedding from
$H^{1}(\Omega)$ to $L_{t}^{q}(\Omega)$ and $L_{s}^{q}(\Omega)$(defined later),
the functional corresponding to \eqref{1.1}
does not satisfy the Palais-Smale condition at large energy level.
Hence it is impossible
to apply the abstract theorem by Schechter and Zou (see\cite{sz}) directly to get the existence of infinitely
many sign-changing solutions for \eqref{1.1}. To overcome this difficulty,
 we intend to use the arguments similar to \cite{cpy,cs1,ds,pw,ww,yy}, i.e., we
 first look at the following perturbed problem:
\begin{equation}\label{1.3}
\left\{%
\begin{array}{ll}
    -\Delta u=\frac{\mu|u|^{2^{*}(t)-2-\epsilon_n}u}{|y|^{t}}
    +\frac{|u|^{2^{*}(s)-2-\epsilon_n}u}{|y|^{s}}+a(x) u, & \hbox{$\text{in}~ \Omega$},\vspace{0.1cm} \\
   u=0,\,\, &\hbox{$\text{on}~\partial \Omega$}, \\
\end{array}%
\right.
\end{equation}
where $\epsilon_n>0$ is small.

Similar to the arguments used in \cite{mb,g,sz}, using an abstract theorem by Schechter and Zou (see\cite{sz}),
we will prove that for each $\ep_n$, \eqref{1.3} has a sequence of sign-changing solutions $\{u_{n,l}\}_{l=1}^\infty$
and the Morse index of $\{u_{n,l}\}$ has a lower bound. Then we can verify that $\{u_{n,l}\}$ converges to a sign-changing
solution of \eqref{1.1}.
However, we need to mention that, due to the appearance of double Hardy-Sobolev-maz'ya terms, the
studying of our problem \eqref{1.1} is greatly different from \cite{mb,g,sz} and some complicated
calculations will be needed for \eqref{1.1}.

Finally, as mentioned before, due to the appearance of the critical terms, the problem \eqref{1.1} exhibits
nonexistence phenomenon.

\begin{thm}\label{thm1.1.1}
Suppose that $N\geq3$, $a(x)\in C^1(\bar{\Om})$ with $a(x)+\frac{1}{2} x\cdot\nabla a\leq0$ for every $x\in\Om$ and
$\Omega$ is star shaped with respect to the origin.
Then \eqref{1.1} does not have any nontrivial solution.
\end{thm}

Throughout this paper, we denote the norm of  $H^{1}_{0}(\Omega)$ by
$\|u\|=\bigl(\int_{\Omega}|\nabla u|^{2}dx\bigl)^{\frac{1}{2}}$;
the norm of $L_{t}^{q}(\Omega)(1\leq q<\infty,0\leq t<2)$ by
$|u|_{q,t,\Omega}=\bigl(\int_{\Omega}\frac{|u|^{q}}{|y|^{t}}dx\bigl)^{\frac{1}{q}}$
and positive constants (possibly different) by $C$.

The organization of the paper is as follows.
 In Section
  \ref{s2}, we prove the existence and the estimate of Morse index of sign-changing solution of
  \eqref{1.3}.
 Using this,
   we will  prove our main results in Section \ref{s3}.

\section{Existence of sign-changing critical points}\label{s2}

In this section, we will prove the existence of sign-changing solutions for the
perturbed compact problem \eqref{1.3} with an estimate on Morse index.
For this purpose,
we mainly use an abstract theorem of Schechter and Zou (see \cite{sz}, Theorem 2).
However, we can not directly apply it due to the presence of the Hardy-Sobolev-Maz'ya terms in
\eqref{1.3}, so we need some precise estimates.

Firstly, we consider the weighted eigenvalue problem:
\begin{equation}\label{2.1}
\left\{%
\begin{array}{ll}
    -\Delta u=\lambda a(x)u, & \hbox{$\text{in}~ \Omega$},\vspace{0.1cm} \\
   u=0,\,\, &\hbox{$\text{on}~\partial \Omega$}. \\
\end{array}%
\right.
\end{equation}
Since $a(x)$ belongs to $ C^{1}(\bar{\Omega})$ and is strictly positive, \eqref{2.1} has infinitely many eigenvalues
$\{\lambda_i\}_{i=1}^\infty$ such that $0<\lambda_1<\lambda_2\leq\lambda_3\leq\cdots\leq\lambda_\ell\leq\cdots$.
Moreover, we can write
\begin{equation}\label{2.2}
\lambda_1=\inf_{u\in H^1_0(\Om)\backslash\{0\}}\frac{\int_{\Omega}|\nabla u|^{2}}{\int_{\Omega}a(x)u^2}.
\end{equation}
Let $\varphi_i$ be the eigenfunction corresponding to $\lambda_i$ and denote
$E_\ell:=span\{\varphi_1,\varphi_1\cdots,\varphi_\ell\}$.
Then $\varphi_1>0$, $E_\ell\subset E_{\ell+1}$ and $H_0^1(\Om)=\overline{\cup_{\ell=1}^\infty E_\ell}$ (see \cite{mw}).

\begin{lem}\label{lem2.1}
Suppose that all the assumptions in Theorem \ref{thm1.1} hold and $\lambda_1\leq1$. Then \eqref{1.1} has infinitely many
sign-changing solutions.
\end{lem}

\begin{proof}
Multiplying $\varphi_1$ to the equation \eqref{1.1} and integrating by parts, we find
$$(\lambda_1-1)\ds\int_{\Omega}a(x)u\varphi_1
=\mu\int_{\Omega}\frac{|u|^{2^{*}(t)-2}u\varphi_1}{|y|^{t}}
+\int_{\Omega}\frac{|u|^{2^{*}(s)-2}u\varphi_1}{|y|^{s}}.$$
So it is easy to see that if $\lambda_1\leq1$, then any nontrivial solution of \eqref{1.1} has to change sign.
Since it follows from Theorem 1.1 in \cite{wyj} that \eqref{1.1} has infinitely many solutions, we have proved our result.
\end{proof}

Using Lemma \ref{lem2.1}, to prove Theorem \ref{thm1.1}, next we just need to discuss $\lambda_1>1$.
Fix $\epsilon_0>0$ small enough and choose a sequence $\epsilon_n\in(0,\epsilon_0)$ such that
$\epsilon_n\downarrow0$ in \eqref{1.3}.

Obviously,
the energy functional corresponding to \eqref{1.3} is
\begin{equation}\label{1.4}
I_{\epsilon_n}(u)=\frac{1}{2}\int_{\Omega}\big(|\nabla u|^{2}-a(x)
u^{2}\big)-\frac{\mu}{2^{*}(t)-\epsilon_n}\int_{\Omega}\frac{|u|^{2^{*}(t)-\epsilon_n}}{|y|^{t}}
-\frac{1}{2^{*}(s)-\epsilon_n}\int_{\Omega}\frac{|u|^{2^{*}(s)-\epsilon_n}}{|y|^{s}}.
\end{equation}
Then $I_{\epsilon_n}(u)\in C^2$ is an even functional and satisfies the Palais Smale
 condition in all energy levels. Moreover, it follows from the symmetric
 mountain pass lemma \cite{ar,r}, \eqref{1.3} has infinitely many solutions.
  More precisely, there are positive numbers $c_{n,l},l=1,2,\cdot\cdot\cdot,$ with $c_{n,l}\rightarrow\infty$
 as $l\rightarrow +\infty$,
 and a solution $u_{n,l}$ for \eqref{1.3} satisfying
 $
I_{\epsilon_n}(u_{n,l})=c_{n,l}.
 $

Recall that the augmented Morse index $m^*(u_{n,l})$ of  $u_{n,l}$ is defined  as
$$
m^*(u_{n,l}):=\max\{dim H: H\subset H_0^1(\Om)\, \text{is a subspace such that}\,
\left \langle I''_{\epsilon_n}(v),v\right\rangle\leq 0,\,\forall v\in H\}.
$$
For each $\epsilon_n\in(0,\epsilon_0)$ fixed, we define
$$
\|u\|_{n,*}:=\mu\Big(\int_{\Omega}\frac{|u|^{2^{*}(t)-\epsilon_n}}{|y|^{t}}\Big)^{\frac{1}{2^*(t)-\epsilon_n}}
+\Big(\int_{\Omega}\frac{|u|^{2^{*}(s)-\epsilon_n}}{|y|^{s}}\Big)^{\frac{1}{2^*(s)-\epsilon_n}},\,\,\forall u\in H_0^1(\Om).
$$
Then from the Hardy-Sobolev-Maz'ya inequality \eqref{1.1.1}, we get
$\|u\|_{n,*}\leq C\|u\|$ for all $u\in H_0^1(\Om)$ for some constant $C>0$ independent of $n$.
Moreover, for fixed $n$, we have $\|u_k-u\|_{n,*}\rightarrow 0$ whenever $u_k\rightharpoonup u$ weakly in $H_0^1(\Om)$.

We write $\mathcal{P}:=\{u\in H_0^1(\Om):\,u\geq 0\}$ for the convex cone of nonnegative functions in $H_0^1(\Om)$
and define for $\rho>0$,
$$
D(\rho):=\{u\in H_0^1(\Om):\,dist(u,\mathcal{P})<\rho\}.
$$
Denote the set of all critical points by
$$
K_n:=\{u\in H_0^1(\Om):\,I'_{\epsilon_n}(u)=0\}.
$$

Note that the gradient $I'_{\epsilon_n}$ is of form $I'_{\epsilon_n}(u)=u-\mathcal{A}_n(u)$,
where $\mathcal{A}_n:\,H_0^1(\Om)\rightarrow H_0^1(\Om)$ is a continuous operator. Now we will
study how the operator $\mathcal{A}_n$ behaves on $D(\rho)$.

\begin{prop}\label{prop2.2}
Let $\lambda_1>1$. Then for any $\rho_{0}>0$ small enough, $\mathcal{A}_n(\pm D(\rho_0))
\subset \pm D(\rho)\subset \pm D(\rho_0)$ for some $\rho\in(0,\rho_0)$. Moreover,
$\pm D(\rho_0)\cap K_n\subset \mathcal{P}$.
\end{prop}

\begin{proof}
Firstly, we observe that $\mathcal{A}_n$ can be decomposed as $\mathcal{A}_n(u)=B(u)+F_n(u)$,
where $B(u), F_n(u)\in H_0^1(\Om)$ are the unique solutions of the following equations respectively
$$
-\Delta(B(u))=a(x)u,\,\,\,\,\,-\Delta(F_n(u))=\frac{\mu|u|^{2^{*}(t)-2-\epsilon_n}u}{|y|^{t}}
    +\frac{|u|^{2^{*}(s)-2-\epsilon_n}u}{|y|^{s}}.
    $$
In other words, $B(u)$ and $F_n(u)$ can be uniquely determined by
\begin{equation}\label{2.4}
\left \langle B(u),v\right\rangle_{H_0^1(\Om)}=\int_{\Omega}a(x)uv,
\end{equation}
and
\begin{equation}\label{2.5}
\left \langle F_n(u),v\right\rangle_{H_0^1(\Om)}
=\mu\int_{\Omega}\frac{|u|^{2^{*}(t)-2-\epsilon_n}uv}{|y|^{t}}
+\int_{\Omega}\frac{|u|^{2^{*}(s)-2-\epsilon_n}uv}{|y|^{s}}.
\end{equation}
Now we claim that if $u\in \mathcal{P}$, then $B(u), F_n(u)\in \mathcal{P}$. To see this,
let $u\in \mathcal{P}$. Then
$$-\int_{\Omega}|\nabla B(u)^-|^2=\left \langle B(u),B(u)^-\right\rangle_{H_0^1(\Om)}
=\int_{\Omega}a(x)uB(u)^-\geq0,$$
which yields $B(u)\in \mathcal{P}$, and similarly $F_n(u)\in\mathcal{P}$.

Not that
$$
\|B(u)\|^2=\left \langle B(u),B(u)\right\rangle_{H_0^1(\Om)}
=\int_{\Omega}a(x)uB(u)\leq \Big(\int_{\Omega}a(x)u^2\Big)^{\frac{1}{2}}
\Big(\int_{\Omega}a(x)B(u)^2\Big)^{\frac{1}{2}}.
$$
So it follows from \eqref{2.2} that
$$
\|B(u)\|^2\leq \frac{1}{\lambda_1}\|u\|\|B(u)\|,
$$
and then
$$
\|B(u)\|\leq \frac{1}{\lambda_1}\|u\|.
$$

For any $u\in H_0^1(\Om)$, we consider $v\in \mathcal{P}$ such that
$dist(u,\mathcal{P})=\|u-v\|$. Then
\begin{equation}\label{2.6}
dist(B(u),\mathcal{P})\leq \|B(u)-B(v)\|\leq\frac{1}{\lambda_1}\|u-v\|
\leq \frac{1}{\lambda_1}dist(u,\mathcal{P}).
\end{equation}
On the other hand, using the equality \eqref{1.1.1}, we find
\begin{align*}
dist(F_n(u),\mathcal{P})\|F_n(u)^-\|
&\leq \|F_n(u)-F_n(u)^+\|\|F_n(u)^-\|=\|F_n(u)^-\|^2\\
&=-\left \langle F_n(u),F_n(u)^-\right\rangle_{H_0^1(\Om)}\\
&=-\mu\int_{\Omega}\frac{|u|^{2^{*}(t)-2-\epsilon_n}uF_n(u)^-}{|y|^{t}}
-\int_{\Omega}\frac{|u|^{2^{*}(s)-2-\epsilon_n}uF_n(u)^-}{|y|^{s}}\\
&\leq \mu\int_{\Omega}\frac{|u|^{2^{*}(t)-2-\epsilon_n}u^-F_n(u)^-}{|y|^{t}}
+\int_{\Omega}\frac{|u|^{2^{*}(s)-2-\epsilon_n}u^-F_n(u)^-}{|y|^{s}}\\
&=\mu\int_{\Omega}\frac{|u^-|^{2^{*}(t)-1-\epsilon_n}F_n(u)^-}{|y|^{t}}
+\int_{\Omega}\frac{|u^-|^{2^{*}(s)-1-\epsilon_n}F_n(u)^-}{|y|^{s}}\\
&\leq \mu\Big(\int_{\Omega}\frac{|u^-|^{2^{*}(t)-\epsilon_n}}{|y|^{t}}\Big)^{\frac{2^{*}(t)-1-\epsilon_n}{2^{*}(t)-\epsilon_n}}
\Big(\int_{\Omega}\frac{|F_n(u)^-|^{2^{*}(t)-\epsilon_n}}{|y|^{t}}\Big)^{\frac{1}{2^{*}(t)-\epsilon_n}}\\
&\quad+\Big(\int_{\Omega}\frac{|u^-|^{2^{*}(s)-\epsilon_n}}{|y|^{s}}\Big)^{\frac{2^{*}(s)-1-\epsilon_n}{2^{*}(s)-\epsilon_n}}
\Big(\int_{\Omega}\frac{|F_n(u)^-|^{2^{*}(s)-\epsilon_n}}{|y|^{s}}\Big)^{\frac{1}{2^{*}(s)-\epsilon_n}}\\
&\leq C \big(|u^-|_{2^{*}(t)-\epsilon_n,t,\Om}^{2^{*}(t)-1-\epsilon_n}
+|u^-|_{2^{*}(s)-\epsilon_n,s,\Om}^{2^{*}(s)-1-\epsilon_n}\big)\|F_n(u)^-\|,
\end{align*}
which implies that for any $u\in H_0^1(\Om)$,
\begin{equation}\label{2.7}
\begin{array}{ll}
dist(F_n(u),\mathcal{P})
\leq C\big(dist(u,\mathcal{P})^{2^{*}(t)-1-\epsilon_n}
+dist(u,\mathcal{P})^{2^{*}(s)-1-\epsilon_n}\big),
\end{array}
\end{equation}
since
\begin{align*}
|u^-|_{2^{*}(t)-\epsilon_n,t,\Om}=\min\limits_{v\in\mathcal{P}}|u-v|_{2^{*}(t)-\epsilon_n,t,\Om}
\leq C\min\limits_{v\in\mathcal{P}}||u-v||=Cdist(u,\mathcal{P}).
\end{align*}

Since $\la_1>1$, we can choose $\ga\in (\frac{1}{\la_1},1)$. Then from \eqref{2.7}, there exists $\rho_0>0$ small enough such that
for any $u\in D(\rho_0)$,
\begin{equation}\label{2.8}
dist(F_n(u),\mathcal{P})
\leq \big(\ga-\frac{1}{\la_1}\big)dist(u,\mathcal{P}).
\end{equation}
Combining \eqref{2.6} and \eqref{2.8}, one has
$$dist(\mathcal{A}_n(u),\mathcal{P})\leq dist(B(u),\mathcal{P})
+dist(F_n(u),\mathcal{P})\leq
\ga dist(u,\mathcal{P}),\,\,\,\forall u\in D(\rho_0).$$

Hence we can get $\mathcal{A}_n(D(\rho_0))\subset D(\ga\rho_0)\subset D(\rho_0)$. Also if ,
$dist(u,\mathcal{P})<\rho_0$ and $I'_{\epsilon_n}(u)=0$, that is, $u=\mathcal{A}_n(u)$, then
$dist(u,\mathcal{P})=dist(\mathcal{A}_n(u),\mathcal{P})\leq \ga dist(u,\mathcal{P})$, which yields
$u\in \mathcal{P}$. Similarly, we can prove $\mathcal{A}_n(-D(\rho_0))\subset -D(\rho)\subset-D(\rho_0)$
for some $\rho\in (0,\rho_0)$ and $-D(\rho)\cap K_n\subset\mathcal{P}$. This completes our proof.
\end{proof}

\begin{prop}\label{prop2.3}
Let $\lambda_1>1$. Then for any $\ell$, $\lim\limits_{\|u\|\rightarrow\infty,u\in E_\ell}I_{\epsilon_n}(u)=-\infty$.
\end{prop}

\begin{proof}
For each $n$, $\|\cdot\|_{n,*}$ defines a norm on $H_0^1(\Om)$. Now since $E_\ell$ is finite dimensional, there exists
a constant $C>0$ such that $\|u\|\leq C\|u\|_{*,n}$ for all $u\in E_\ell$. So,
\begin{align*}
I_{\epsilon_n}(u)&\leq \frac{1}{2}\int_{\Omega}|\nabla u|^{2}-\frac{\mu}{2^{*}(t)-\epsilon_n}\int_{\Omega}\frac{|u|^{2^{*}(t)-\epsilon_n}}{|y|^{t}}
-\frac{1}{2^{*}(s)-\epsilon_n}\int_{\Omega}\frac{|u|^{2^{*}(s)-\epsilon_n}}{|y|^{s}}\\
&\leq \frac{1}{2}\|u\|^2-C\|u\|^{2^{*}(s)-\epsilon_n}.
\end{align*}
As a result,
$$
\lim\limits_{\|u\|\rightarrow\infty,u\in E_\ell}I_{\epsilon_n}(u)=-\infty,
$$
since $2^{*}(s)-\epsilon_n>2$. Then we have proved the desired result.
\end{proof}

\begin{prop}\label{prop2.4}
For any $\alpha_1, \alpha_2>0$, there exists $\alpha_3$ depending on $\alpha_1$ and $\alpha_2$ such that
$\|u\|\leq \al_3$ for all $u\in I_{\epsilon_n}^{\al_1} \cap \{u\in H_0^1(\Om):\,\|u\|_{*,n}\leq \al_2\}$, where
$I_{\epsilon_n}^{\al_1}=\{u\in H_0^1(\Om):\,I_{\epsilon_n}(u)\leq \al_1\}$.
\end{prop}

\begin{proof}
Noting that $\|u\|_{*,n}\leq \al_2$, we have
$$
\mu^{2^{*}(t)-\epsilon_n}\int_{\Omega}\frac{|u|^{2^{*}(t)-\epsilon_n}}{|y|^{t}}+
\Big(\int_{\Omega}\frac{|u|^{2^{*}(s)-\epsilon_n}}{|y|^{s}}\Big)^{\frac{2^{*}(t)-\epsilon_n}{2^{*}(s)-\epsilon_n}}
\leq \al_2^{2^{*}(t)-\epsilon_n},
$$
and then
$$
\mu \int_{\Omega}\frac{|u|^{2^{*}(t)-\epsilon_n}}{|y|^{t}}
+\int_{\Omega}\frac{|u|^{2^{*}(s)-\epsilon_n}}{|y|^{s}}\leq C\al_2^{2^{*}(t)-\epsilon_n}.
$$
So from the fact that $I_{\epsilon_n}(u)\leq \al_1$, we find
\begin{equation}\label{2.9}
\frac{1}{2}\int_{\Omega}\big(|\nabla u|^{2}-a(x)
u^{2}\big)+\mu\Big(\frac{1}{2^*(s)-\epsilon_n}-\frac{1}{2^*(t)-\epsilon_n}\Big)
\int_{\Omega}\frac{|u|^{2^{*}(t)-\epsilon_n}}{|y|^{t}} \leq\al_1+ C\al_2^{2^{*}(t)-\epsilon_n}.
\end{equation}
Using \eqref{2.9} and \eqref{2.2}, we have
$$
\frac{1}{2}\Big(1-\frac{1}{\la_1}\Big)\int_{\Omega}|\nabla u|^{2}+\mu\Big(\frac{1}{2^*(s)-\epsilon_n}-\frac{1}{2^*(t)-\epsilon_n}\Big)
\int_{\Omega}\frac{|u|^{2^{*}(t)-\epsilon_n}}{|y|^{t}} \leq C,
$$
which implies our desired result.\emph{}
\end{proof}

\begin{lem}\label{lem2.5}
Let $\la_1>1$. Then for each $n$, \eqref{1.3} has infinitely many sign-changing solutions $\{u_{n,l}\}_{l=1}^\infty$ such that
for each $l$, the sequence $\{u_{n,l}\}$ is bounded in $H_0^1(\Om)$ and the augmented Morse index of $\{u_{n,l}\}$ is greater than
or equal to $l$.
\end{lem}

\begin{proof}
Firstly, it follows from Propositions \ref{prop2.2}, \ref{prop2.3} and \ref{prop2.4} that
$I_{\epsilon_n}$ satisfies all the conditions of Theorem 2 in \cite{sz}. Thus $I_{\epsilon_n}$ has a sign-changing
critical point $u_{n,l}\in H_0^1(\Om)$ at the level $c_{n,l}$ with $c_{n,l}\leq \sup_{E_{l+1}}I_{\epsilon_n}$
and $m^*(u_{n,l})\geq l$. Now it remains to show that the sequence $\{u_{n,l}\}$ is bounded for each $l$.

Claim: there exists $T_1>0$ independent of $n$ and $l$ such that
$$
c_{n,l}\leq T_1\la_{l+1}^{\frac{2^*(s)-\epsilon_0}{2(2^*(s)-\epsilon_0-2)}}.
$$

To see this, since $2^*(s)-\epsilon_0>2$, we have
\begin{equation}\label{2.10}
\|u\|^2\leq C\la_{l+1}|u|_2^2\leq C\la_{l+1}|u|^2_{2^*(s)-\epsilon_0,s,\Om},\,\, u\in E_{l+1},
\end{equation}
where $C>0$ is a constant independent of $n,l$.
Moreover, from the fact that $2^*(s)-\epsilon_0<2^*(s)-\epsilon_n$, we get that there exist $C_1, C_2>0$ independent of
$n, l$ such that
$$
|u|_{2^*(s)-\epsilon_0,s,\Om}\leq C_1|u|_{2^*(s)-\epsilon_n,s,\Om}+C_2.
$$
Therefore,
\begin{equation}\label{2.11}
\begin{array}{ll}
I_{\epsilon_n}(u)&\leq \ds\frac{1}{2}\ds\int_{\Omega}|\nabla u|^{2}
-\ds\frac{1}{2^{*}(s)-\epsilon_n}\ds\int_{\Omega}\frac{|u|^{2^{*}(s)-\epsilon_n}}{|y|^{s}}\vspace{0.15cm}\\
&\leq \ds\frac{1}{2}\|u\|^2
-C_3\ds\int_{\Omega}\frac{|u|^{2^{*}(s)-\epsilon_0}}{|y|^{s}}+C_4,
\end{array}
\end{equation}
where $C_3,C_4>0$ independent of $n,l$.

So from \eqref{2.10} and \eqref{2.11}, we obtain for any $u\in E_{l+1}$,
\begin{equation}\label{2.11.1}
\begin{array}{ll}
I_{\epsilon_n}(u)&\leq \frac{1}{2}\|u\|^2
-C_5\la_{l+1}^{-\frac{2^*(s)-\epsilon_0}{2}}\|u\|^{2^*(s)-\epsilon_0}+C_4\vspace{0.12cm}\\
&\leq C_6\la_{l+1}^{\frac{2^*(s)-\epsilon_0}{2(2^*(s)-\epsilon_0-2)}}+C_4
\leq T_1\la_{l+1}^{\frac{2^*(s)-\epsilon_0}{2(2^*(s)-\epsilon_0-2)}},
\end{array}
\end{equation}
where $C_i>0(i=1,\cdots,6)$ and $T_1>0$ independent of $n,l$.

Moreover, we have
\begin{equation}\label{2.11.2}
\begin{array}{ll}
I_{\epsilon_n}(u_{n,l})&=\ds\frac{1}{2}\int_{\Omega}\big(|\nabla u_{n,l}|^{2}-a(x)
u_{n,l}^{2}\big)-\frac{\mu}{2^{*}(t)-\epsilon_n}\int_{\Omega}\frac{|u_{n,l}|^{2^{*}(t)-\epsilon_n}}{|y|^{t}}
-\frac{1}{2^{*}(s)-\epsilon_n}\int_{\Omega}\frac{|u_{n,l}|^{2^{*}(s)-\epsilon_n}}{|y|^{s}}
\vspace{0.12cm}\\
&=\Big(\ds\frac{1}{2}-\frac{1}{2^{*}(t)-\epsilon_n}\Big)\mu\int_{\Omega}\frac{|u_{n,l}|^{2^{*}(t)-\epsilon_n}}{|y|^{t}}
+\Big(\frac{1}{2}-\frac{1}{2^{*}(s)-\epsilon_n}\Big)\int_{\Omega}\frac{|u_{n,l}|^{2^{*}(s)-\epsilon_n}}{|y|^{s}}
\vspace{0.12cm}\\
&\geq 0.
\end{array}
\end{equation}

Then we conclude from \eqref{2.11.1} and \eqref{2.11.2}
that
$I_{\epsilon_n}(u_{n,l})\in \big(0, T_1\la_{l+1}^{\frac{2^*(s)-\epsilon_0}{2(2^*(s)-\epsilon_0-2)}}\big]$.
So,
\begin{align*}
I_{\epsilon_n}(u_{n,l})&= I_{\epsilon_n}(u_{n,l})-
\frac{1}{2^*(s)-\epsilon_n}\left \langle I'_{\epsilon_n}(u_{n,l}),u_{n,l}\right\rangle\\
&=\Big(\frac{1}{2}-\frac{1}{2^*(s)-\epsilon_n}\Big)\int_{\Omega}\big(|\nabla u_{n,l}|^{2}-a(x)u^{2}\big)\\
&\quad+\mu\Big(\frac{1}{2^*(s)-\epsilon_n}-\frac{1}{2^*(t)-\epsilon_n}\Big)
\int_{\Omega}\frac{|u_{n,l}|^{2^{*}(t)-\epsilon_n}}{|y|^{t}}\\
&\geq \Big(\frac{1}{2}-\frac{1}{2^*(s)-\epsilon_0}\Big)\int_{\Omega}\Big(|\nabla u_{n,l}|^{2}-a(x)u^{2}\Big)\\
&\geq \Big(\frac{1}{2}-\frac{1}{2^*(s)-\epsilon_0}\Big)\big(1-\frac{1}{\la_1}\big)\int_{\Omega}|\nabla u_{n,l}|^{2},
\end{align*}
which completes our proof.
\end{proof}

\section{proof of the main results }\label{s3}

In this section, we will prove our main results. Firstly,
in order to prove the existence result of infinitely many sign changing solutions,
 we introduce the following result given in \cite{wyj}.

\begin{thm}\label{thm1.2}
Suppose that $a(x)>0$ and
$\Omega$ satisfies the same geometric conditions as imposed in Theorem \ref{thm1.1}.
If $N>6+t$ when $\mu>0,$ and $N>6+s$ when $\mu=0,$ then for any $u_{n},$ which is a solution of \eqref{1.3}
with $\epsilon_{n}\rightarrow 0,$ satisfying $\|u_{n}\|\leq C$
for some constant independent of $n,u_{n}$ converges strongly in $H_{0}^{1}(\Omega)$
as $n\rightarrow +\infty.$
\end{thm}

Now we are ready to prove Theorems \ref{thm1.1} and \ref{thm1.1.1}.

\noindent \textbf{ Proof of Theorem \ref{thm1.1}:}
We divide our proof into two steps.

Step~1: It follows from Lemma \ref{lem2.5} and Theorem \ref{thm1.2} that $u_{n,l}\rightarrow u_l$
in $H_0^1(\Om)$ as $n\rightarrow\infty$ and $\{u_l\}_{l=1}^\infty$ is a sequence of solutions to
\eqref{1.1} with energy $c_l\in\big[0,T_1\la_{l+1}^{\frac{2^*(s)-\epsilon_0}{2(2^*(s)-\epsilon_0-2)}}\big]$.

Now we claim that $u_l$ is still sign-changing solution to \eqref{1.1}. In fact,
since $I'_{\epsilon_n}(u_{n,l})=0$, we have
 $$
 \int_{\Omega}\big(|\nabla u_{n,l}^{\pm}|^{2}-a(x)|u_{n,l}^{\pm}|^{2}\big)=
 \mu\int_{\Omega}\frac{|u_{n,l}^{\pm}|^{2^{*}(t)-\epsilon_n}}{|y|^{t}}
 +\int_{\Omega}\frac{|u_{n,l}^{\pm}|^{2^{*}(s)-\epsilon_n}}{|y|^{s}},
 $$
and then from \eqref{2.2},
\begin{equation}\label{3.1}
\begin{array}{ll}
\Big(1-\ds\frac{1}{\la_1}\Big)\|\nabla u_{n,l}^{\pm}\|^2
&\leq \mu\ds\int_{\Omega}\frac{|u_{n,l}^{\pm}|^{2^{*}(t)-\epsilon_n}}{|y|^{t}}
 +\ds\int_{\Omega}\frac{|u_{n,l}^{\pm}|^{2^{*}(s)-\epsilon_n}}{|y|^{s}}\vspace{2mm}\\
 &\leq \mu\|\nabla u_{n,l}^{\pm}\|^{2^{*}(t)-\epsilon_n}+\|\nabla u_{n,l}^{\pm}\|^{2^{*}(s)-\epsilon_n}.
\end{array}
\end{equation}
So we can infer from \eqref{3.1} that $\|\nabla u_{n,l}^{\pm}\|\geq C>0$ for some $C$ independent of $n$.
This in turn implies that $\|\nabla u_{l}^{\pm}\|\geq C>0$.

Step~2: To complete the proof, we only need to prove that infinitely many $u_l$'s are different.
it is sufficient to check
that the energy of $u_l$ goes to infinity as $l\rightarrow \infty$.
Here we make a contradiction argument. Assume that $\lim\limits_{l\rightarrow \infty}c_l=c'<+\infty.$
 Then for each $l$, we can find $n_l>l$ such that
$|c_{n_l,l}-c_l|<\frac{1}{l}$ and $\lim\limits_{l\rightarrow \infty}c_{n_l,l}=\lim\limits_{l\rightarrow \infty}c_l=c'<+\infty$.
Therefore $u_{n_l,l}$ is bounded in $H_0^1(\Om)$ and it follows from Theorem \ref{thm1.2}
 that $u_{n_l,l}$ converges in $H_0^1(\Om)$ and
the augmented Morse index of $u_{n_l,l}$ remains bounded. This contradicts to the fact that
$m^*(u_{n_l,l})\geq l$ and completes the proof.

\noindent \textbf{ Proof of Theorem \ref{thm1.1.1}:}
To prove this theorem, we will mainly apply the Pohozaev identity motivated by \cite{bs}.

For $\epsilon>0$ and $R>0$, define $\varphi_{\epsilon,R}(x)=\varphi_\epsilon(x)\phi_R(x)$, where
$\varphi_\epsilon(x)=\varphi(\frac{|x|}{\epsilon})$, $\phi_R(x)=\phi(\frac{|x|}{R})$,
$\var$ and $\phi$ are smooth functions in $\R$ satisfying $0\leq\var, \phi\leq 1$ with supports of
$\var, \phi$ in $(1,+\infty)$ and $(-\infty, 2)$ respectively and $\var(t)=1$ for $t\geq2$ and
$\phi=1$ for $t\leq1$.

Assume that \eqref{1.1} has a nontrivial solution $u$. Then $(x\cdot\nabla u)\varphi_{\epsilon,R}(x)\in C_c^2(\Om)$.
Multiplying \eqref{1.1} by this text function and integrating by parts, we find
\begin{equation}\label{3.2}
\begin{array}{ll}
&\ds\int_{\Omega}\nabla u\nabla((x\cdot\nabla u)\varphi_{\epsilon,R})
-\ds\int_{\partial \Om}\frac{\pa u}{\pa \nu}(x\cdot\nabla u)\varphi_{\epsilon,R}dS\\
&=\mu \ds\int_{\Omega}\frac{|u|^{2^{*}(t)-2}u}{|y|^{t}}(x\cdot\nabla u)\varphi_{\epsilon,R}
+\ds\int_{\Omega}\frac{|u|^{2^{*}(s)-2}u}{|y|^{s}}(x\cdot\nabla u)\varphi_{\epsilon,R}
\\
&+\ds\int_{\Omega}a(x)u(x\cdot\nabla u)\varphi_{\epsilon,R}.
\end{array}
\end{equation}
Firstly, we can simplify the RHS of \eqref{3.2} as follows
\begin{align*}
&\mu \ds\int_{\Omega}\frac{|u|^{2^{*}(t)-2}u}{|y|^{t}}(x\cdot\nabla u)\varphi_{\epsilon,R}
+\ds\int_{\Omega}\frac{|u|^{2^{*}(s)-2}u}{|y|^{s}}(x\cdot\nabla u)\varphi_{\epsilon,R}
+\ds\int_{\Omega}a(x)u(x\cdot\nabla u)\varphi_{\epsilon,R}\\
&=-\mu(\frac{N-2}{2})\int_{\Omega}\frac{|u|^{2^{*}(t)}}{|y|^{t}}\varphi_{\epsilon,R}
-\frac{\mu}{2^*(t)}
\frac{|u|^{2^{*}(t)}}{|y|^{t}}\big(x\cdot(\var_\epsilon\nabla \phi_R+\phi_R\nabla\var_\epsilon)\big)\\
&\quad-\frac{N-2}{2}\int_{\Omega}
\frac{|u|^{2^{*}(s)}}{|y|^{s}}\varphi_{\epsilon,R}
-\frac{1}{2^*(s)}\frac{|u|^{2^{*}(s)}}{|y|^{s}}\big(x\cdot(\var_\epsilon\nabla \phi_R+\phi_R\nabla\var_\epsilon)\big)\\
&\quad-\frac{N}{2}\int_{\Omega}a(x)u^2\varphi_{\epsilon,R}-\frac{1}{2}
\int_{\Omega}a(x)u^2\big(x
\cdot(\var_\epsilon\nabla \phi_R+\phi_R\nabla\var_\epsilon)\big)\\
&\quad-\frac{1}{2}\int_{\Omega}u^2(x\cdot\nabla a)\varphi_{\epsilon,R}.
\end{align*}
Since $|x\cdot(\var_\epsilon\nabla \phi_R+\phi_R\nabla\var_\epsilon)|\leq C$, by using the dominated convergence theorem, we have
\begin{equation}\label{3.3}
\lim\limits_{R\rightarrow \infty}\lim\limits_{\epsilon\rightarrow 0}RHS
=-\frac{N-2}{2}\Big(\mu \int_{\Omega}\frac{|u|^{2^{*}(t)}}{|y|^{t}}
+\int_{\Omega}\frac{|u|^{2^{*}(s)}}{|y|^{s}}\Big)
-\frac{N}{2}\int_{\Omega}a(x)u^2-\frac{1}{2}\int_{\Omega}u^2(x\cdot\nabla a).
\end{equation}

On the other hand, by the calculation in \cite{bs}, we estimate LHS of \eqref{3.2} as
\begin{equation}\label{3.4}
\lim\limits_{R\rightarrow \infty}\lim\limits_{\epsilon\rightarrow 0}LHS
=-\frac{N-2}{2}\int_{\Omega}|\nabla u|^{2}-\frac{1}{2}\ds\int_{\partial \Om}\big(\frac{\pa u}{\pa \nu}\big)^2(x\cdot\nu)dS
\end{equation}
Therefore substituting \eqref{3.3} and \eqref{3.4} in \eqref{3.1} and using \eqref{1.1}, one has
$$
-\int_{\Omega}\big(a(x)+\frac{1}{2}x\cdot\nabla a\big)u^2+\frac{1}{2}\ds\int_{\partial \Om}\big(\frac{\pa u}{\pa \nu}\big)^2(x\cdot\nu)dS=0,
$$
which yields $u=0$ in $\Om$ by the principle of unique continuation. This completes the proof.

{\bf Acknowledgements:}
  This paper was partially supported by NSFC (No.11301204; No.11371159), self-determined research funds of CCNU from the colleges' basic research and operation of MOE (CCNU14A05036) and the excellent doctorial dissertation cultivation grant from Central
China Normal University (2013YBZD15).

\end{document}